\documentclass{article}
\usepackage{}
\usepackage{amsfonts}
\usepackage{amsthm,amssymb,amsmath,amscd}
\usepackage{wasysym}
\usepackage[all]{xy}
\usepackage{mathrsfs}
\usepackage{multirow}
\usepackage{txfonts}
\usepackage{color}
\usepackage{url}

\usepackage{lscape}

\DeclareMathAlphabet{\mathcal}{OMS}{cmsy}{m}{n}
\DeclareSymbolFont{largesymbols}{OMX}{cmex}{m}{n}

\objectmargin{1pt}

\newtheorem{Def}{Definition}[section]
\newtheorem{Prop}[Def]{Proposition}
\newtheorem{Theo}[Def]{Theorem}
\newtheorem{Lem}[Def]{Lemma}
\newtheorem{Koro}[Def]{Corollary}

\newtheorem{Remark}[Def]{Remark}
\newtheorem{exm}[Def]{Example}
\newcommand{\defCategory}[2]{
  \newcommand{#1}{#2\defvariable}
  }

\newcommand{\defvariable}[2][]{
\if\relax\detokenize{#1}\relax  %if the first arg is empty
   \if\relax\detokenize{#2}\relax
    \else
    ({#2})
    \fi
\else
  ^{{#1}}({#2})
\fi
 }

\def\cpx#1{{#1}^{\bullet}}

 \defCategory{\C}{\mathscr{C}}
 \defCategory{\K}{\mathscr{K}}
 \defCategory{\D}{\mathscr{D}}

 \def\Kb#1{\K[b]{#1}}
 \def\Kf#1{\K[-]{#1}}
 \def\Kz#1{\K[+]{#1}}

 \def\Db#1{\D[b]{#1}}
 \def\Df#1{\D[-]{#1}}
 \def\Dz#1{\D[+]{#1}}
 \def\E{{\rm E}}

 \DeclareMathOperator{\add}{add}
 \DeclareMathOperator{\Add}{Add}

 \DeclareMathOperator{\Hom}{Hom}
 \DeclareMathOperator{\HOM}{\mathscr{H}\!\!{\it om}}

 \DeclareMathOperator{\otimesL}{\stackrel{\mathbf{L}}{\otimes}}
 
 \DeclareMathOperator{\thick}{thick}
 \DeclareMathOperator{\cone}{cone}
 \DeclareMathOperator{\per}{per}

\def\Hom{{\rm Hom}}

\def\HOM{{\mathscr{H}\!\!{\it om}}}

\def\rHom{{\bf R}\HOM}
\def\RHom{{\bf R}\Hom}

\def\otimesL{\stackrel{\bf L}{\otimes}}

\def\Tria{{\rm Tria}}

\def\rModcat#1{{\rm Mod}\mbox{-}#1}

\def\modcat#1{#1\mbox{-}{\rm mod}}

\def\rGr#1{{\rm Gr}\mbox{-}#1}
\def\rModcat#1{{\rm Mod}\mbox{-}#1}
\def\rmodcat#1{{\rm mod}\mbox{-}#1}
\def\rpModcat#1{{\rm Proj}\mbox{-}#1}
\def\rpmodcat#1{{\rm proj}\mbox{-}#1}
\def\rstmodcat#1{{\underline{\rm mod}}\mbox{-}#1}

\newcommand{\lra}{\longrightarrow}

\newcommand{\ra}{\rightarrow}

\newcommand{\hra}{\hookrightarrow}
\newcommand{\lraf}[1]{\stackrel{#1}{\lra}}

\newcommand{\hraf}[1]{\stackrel{#1}{\hra}}

\title{Differential graded endomorphism algebras, cohomology rings and derived equivalences}
\author{\sc SHENGYONG PAN, ZHEN PENG and JIE ZHANG$^*$}
\date{}

\begin{document}
\maketitle

\begin{abstract}
 In this paper, we will consider derived equivalences for differential graded endomorphism algebras by Keller's approaches. First we construct derived equivalences of differential graded algebras which are endomorphism algebras of the objects  from a triangle in the homotopy category of differential graded algebras.
We also obtain derived equivalences of differential graded endomorphism algebras from a standard derived equivalence of finite dimensional algebras. Moreover, under some conditions, the cohomology rings of these differential graded endomorphism algebras are also derived equivalent. Then we give an affirmative answer to a problem of Dugas \cite{Dugas2015} in some special case.

\end{abstract}

\renewcommand{\thefootnote}{\alph{footnote}}
\setcounter{footnote}{-1} \footnote{ $^*$ Corresponding author.
Email: jiezhang@bit.edu.cn.}
\renewcommand{\thefootnote}{\alph{footnote}}
\setcounter{footnote}{-1} \footnote{2010 Mathematics Subject
Classification: primary 18E30, 16G10;16S10,18G15.}
\renewcommand{\thefootnote}{\alph{footnote}}
\setcounter{footnote}{-1} \footnote{Keywords: derived equivalence; differential graded endomorphism algebra; cohomology ring; standard derived equivalence.}

%\tableofcontents

\section{Introduction}

Derived equivalences were introduced by Grothendieck and Verdier and play an important
role nowadays in many branches of
algebraic geometry, in algebraic analysis, noncommutative algebraic geometry, representation theory, mathematical physics.
They capture much of the homological information of an abelian category. For example, the Grothendieck group of an abelian group
is preserved under derived equivalences, and derived equivalent algebras have the same Hochschild cohomology.

Important explicit derived equivalences were constructed by Happel \cite{Happel1988} in tilting theory and by
Rickard \cite{Rickard1989a} in tilting complexes. Keller's ICM talk \cite{Keller2006}, for instance, justifies
derived equivalences by the unbounded derived category of differential graded categories \cite{Keller1994}.
There are two natural directions of research: How to give systematic methods for constructing derived
equivalences, and how to produce new derived equivalences
from given ones.

Recently, Hu and Xi constructed  \cite{Hu2013} the derived equivalence for $\Phi$-Auslander-Yoneda algebras.
We want to mention that derived equivalence for $\Phi$-Cohen-Macaulay Auslander-Yoneda algebras  were also constructed in \cite{Pan2014b}. Recall that the $\Phi$-Auslander-Yoneda algebras \cite[Section 3]{Hu2013}, include, for instance, Auslander algebras, generalized
Yoneda algebras and certain trivial extensions. By the definition, $\Phi$-Auslander-Yoneda algebras can be characterized as the
cohomology rings of some dg algebras.  Hence, it is important to
consider derived equivalences for some dg algebras.
It is due to Dugas \cite[Section 7]{Dugas2015}, who proposed the following problem.\\

\noindent{\bf Problem.} Let $\Lambda$ and $\Gamma$ be differential graded algebras. If the cohomology rings $H^*(\Lambda)$ and $H^*(\Gamma)$ are derived equivalent,
are the differential graded algebras $\Lambda$ and $\Gamma$  derived equivalent or not.

Motivated by the above problem, we look for in this paper a general way to get derived equivalences of differential graded algebras in two directions.

By considering a triangle
$$X\longrightarrow Y\longrightarrow Z\longrightarrow X[1]$$
in the homotopy category $\K{A}$ of a differential graded algebra $A$, we show that there is a derived equivalence between differential graded endomorphism algebras
$\HOM(X\oplus Y ,X\oplus Y)$ and $\HOM(Y\oplus Z,Y\oplus Z)$.
See Theorem \ref{4.1} for more details. Moreover, under some mild conditions, we get the cohomology rings of these differential graded algebra are also derived equivalent. This gives
an affirmative answer to the above problem in some special case.
Note that it is still an open question whether a derived equivalence of the cohomology rings implies the differential graded algebras are derived equivalent.

We also study a standard derived equivalence $F:\Db{\rmodcat{B}}\lra \Db{\rmodcat{C}}$
of two finite dimensional algebras $B$ and $C$. By considering the additive functor $\underline{F}: \rstmodcat{B}\ra\rstmodcat{C}$ introduced in \cite{Hu2010,Hu2017},
we get a derived equivalence between differential graded algebras
$\RHom_{B}(B\oplus X,B\oplus X)$ and $\RHom_{C}(C\oplus\underline{F}(X),C\oplus\underline{F}(X)),$
where $X$ is a $B$-module (See Theorem \ref{5.2}). If $X\in {}^{\perp}B$, then the cohomology rings $H^{\mathbb{Z}}(A\oplus X)$  and $H^{\mathbb{Z}}(B\oplus\underline{F}(X))$ of $\RHom_{A}(A\oplus X,A\oplus X)$ and $\RHom_{B}(B\oplus\underline{F}(X),B\oplus\underline{F}(X))$, respectively, are derived
equivalent by \cite[Theorem 1.1]{Pan2014b}. This gives
an affirmative answer to a question of Pan \cite[Section 4]{Pan2014b} which is in spirit to a question of Dugas \cite{Dugas2015}.

\smallskip
This paper is organized as follows. In Section \ref{section-preliminaries}, we recall some basic definitions and facts of derived categories and derived equivalences.
Section \ref{dg modules} is devoted to collecting differential graded modules, their derived categories and studying the differential graded bimodules for later use. The construction of derived equivalences from the triangles in the homotopy category of differential graded algebras will be given in Section \ref{triangles induced derived euivalences}. Finally, in Section \ref{section 5}, we construct derived equivalences for differential graded algebras from standard derived equivalences of finite dimensional algebras.

\section{Preliminaries}\label{section-preliminaries}
In this section, we recall some basic definitions and collect some basic facts of derived categories and derived equivalences.

In this paper, we fix a commutative ring $k$ with identity. All algebras are $k$-algebras, and functors are $k$-functors.  The composite of two morphisms $f: X\ra Y$ and $g: Y\ra Z$ in a category ${\cal C}$ will be denoted by $gf$. If $f: X\ra Y$ is a map between two sets, then the image of an element  $x\in X$ will be denoted by $f(x)$. The composite of two functors $F: {\cal C}\ra {\cal D}$ and $G: {\cal D}\ra {\cal E}$ will be denoted by $GF$. For each object $X$ in ${\cal C}$, we write $F(X)$ for the corresponding object in ${\cal D}$, and for each morphism $f: X\ra Y$ in ${\cal C}$ we write $F(f)$ for the corresponding morphism in ${\cal D}$ from $F(X)$ to $F(Y)$.

For an object $M$ in a $k$-category ${\cal C}$, we use $\add(M)$ to denote the full subcategory of ${\cal C}$ consisting of direct summands of finite direct sums of copies of $M$. If ${\cal C}$ admits infinite coproducts, then $\Add(M)$ means the full subcategory of ${\cal C}$ consisting of direct summands of all coproducts of copies of $M$.

\medskip
 For generality, we shall consider arbitrary $k$-algebras.  All modules will be unitary right modules. Let $A$ be a $k$-algebra. The category of all right $A$-modules will be denoted by $\rModcat{A}$.  We use $\rmodcat{A}$ to denote the full subcategory of $\rModcat{A}$ consisting of  finitely presented $A$-modules, that is, $A$-modules $X$ admitting a projective presentation $P_1\ra P_0\lra X\ra 0$ with $P_i$ finitely generated projective for $i=0, 1$. The category of all projective $A$-modules is denoted by $\rpModcat{A}$, and the category of all finitely generated projective $A$-modules is denoted by $\rpmodcat{A}$.

\medskip

 Let $A$ be a $k$-algebra,  a complex $\cpx{X}$ over $A$ is a sequences $d_X^i$ between $A$-modules
 $X^i$:
 $$\cdots\lra X^{i-1}\lraf{d_X^{i-1}}X^i\lraf{d_X^i}X^{i+1}\lraf{d_X^{i+1}}\cdots$$
 such that $d_X^id_X^{i+1}=0$ for all $i\in\mathbb{Z}$. The category of complexes of $A$-modules, in which morphisms are chain maps, is denoted by $\C{\rModcat{A}}$, and the corresponding homotopy category is denoted by $\K{\rModcat{A}}$. We write $\D{\rModcat{A}}$ for the derived category of $\rModcat{A}$.
We also write $\Kb{\rModcat{A}}$, $\Kf{\rModcat{A}}$ and $\Kz{\rModcat{A}}$ for the full subcategories of $\K{\rModcat{A}}$ consisting of bounded complexes, complexes bounded above, and complexes bounded below, respectively. Denote by $\Db{\rModcat{A}}$, $\Df{\rModcat{A}}$ and $\Dz{\rModcat{A}}$  the full subcategories of $\D{\rModcat{A}}$ consisting of bounded complexes, complexes bounded above, and complexes bounded below, respectively.

Note that there is a fully-faithful functor $\rModcat{A}\lra \D{\rModcat{A}}$ by viewing an $A$-module as a complex in $\D{\rModcat{A}}$ concentrated in degree zero.

\medskip
The homotopy category of an additive category, and the derived category of an abelian category are both triangulated categories.  For basic facts on triangulated categories, we refer to Neeman's book \cite{Neeman2001}.  The shift functor of a triangulated category will be denoted by $[1]$ in this paper.  In the homotopy category, or the derived category of an abelian category, the shift functor acts on a complex by moving the complex to the left by one degree, and changing the sign of the differentials.

\smallskip
Two algebras $A$ and $B$ are said to be {\em derived equivalent} if one of the following equivalent conditions holds

\smallskip
(1).  $\D{\rModcat{A}}$ and $\D{\rModcat{B}}$ are equivalent as triangulated categories.

(2). $\Df{\rModcat{A}}$ and $\Df{\rModcat{B}}$ are equivalent as triangulated categories.

(3). $\Db{\rModcat{A}}$ and $\Db{\rModcat{B}}$ are equivalent as triangulated categories.

(4). $\Kb{\rpModcat{A}}$ and $\Kb{\rpModcat{B}}$ are equivalent as triangulated categories.

(5). $\Kb{\rpmodcat{A}}$ and $\Kb{\rpmodcat{B}}$ are equivalent as triangulated categories.

(6). There is a complex $\cpx{T}$ in $\Kb{\rpmodcat{A}}$
satisfying the conditions:

  \quad \quad (a). $\Hom_{\Kb{\rpmodcat{A}}}(\cpx{T},\cpx{T}[n])= 0$ for all $n\neq 0$,

   \quad \quad (b). $\add(\cpx{T})$ generates $\Kb{\rpmodcat{A}}$ as a triangulated category,

\quad\quad  such that the endomorphism algebra of $\cpx{T}$ in $\Kb{\rpmodcat{A}}$ is isomorphic to $B$.

\medskip
{\parindent-0pt For the } proof that the above conditions are indeed equivalent, we refer to \cite{Rickard1989a,Keller1994}.
If the algebras $A$ and $B$ are left coherent rings, then the above equivalent conditions are further equivalent to the following condition.

\medskip
(7). $\Db{\rmodcat{A}}$ and $\Db{\rmodcat{B}}$ are equivalent as triangulated categories.

\medskip
{\parindent=0pt Note that} in this case the category $\rmodcat{A}$ of finitely presented $A$-modules is an abelian category.  A complex $\cpx{T}$ satisfying the conditions (6.a) and (6.b) above is called a {\em tilting complex}.

A triangle equivalence functor $F: \D{\rModcat{A}}\ra \D{\rModcat{B}}$ is called a {\em derived equivalence}. It is well-known that the image $F(A)$ is isomorphic in $\Db{\rModcat{B}}$ to a tilting complex, and there is a tilting complex $\cpx{T}$ over $A$ such that $F(\cpx{T})$ is isomorphic to $B$ in $\D{\rModcat{B}}$.  The complex $\cpx{T}$ is called an {\em associated tilting complex } of $F$.  The following is an easy lemma for the associated tilting complexes. Its proof can be found, for example, in \cite[Lemma 2.1]{Hu2010}

\begin{Lem}  Consider a derived equivalence
$F: \D{\rModcat{A}}\lra \D{\rModcat{B}}$ of two $k$-algebras $A$ and $B$. Suppose that $F(A)$ is isomorphic in $\D{\rModcat{B}}$ to a complex $\cpx{\bar{T}}\in\Kb{\rpmodcat{B}}$  of the form
$$0\lra \bar{T}^0\lra \bar{T}^1\lra\cdots\lra\bar{T}^n\lra 0$$
for some $n\geq 0$. Then $F^{-1}(B)$ is isomorphic  in $\D{\rModcat{A}}$ to a complex $\cpx{T}\in\Kb{\rpmodcat{A}}$ of the form
$$0\lra T^{-n}\lra\cdots\lra T^{-1}\lra T^0\lra 0.$$
\label{lemma-tiltCompForm}
\end{Lem}

\section{Differential graded algebras and their derived categories}\label{dg modules}

In this section, we collect some background on differential graded algebras, their derived categories of differential graded modules and differential graded bimodules, which will be used all throughout this paper.
Interpreting differential graded algebras as differential graded
categories with just one object, the material is a particular case of the development in \cite{Keller1994,Keller2006,Nicolas2014b,Chen2015a}.

\subsection{Differential graded algebras and differential graded modules}
Recall that a differential graded
$k$-algebra (=dg algebra) is a $\mathbb{Z}$-graded associative
$k$-algebra $A=\displaystyle\oplus_{i\in\mathbb{Z}} A^i$ with a
differential $d:A^i\ra A^{i+1}$ satisfying the graded Leibniz rule
$$d(ab)=d(a)b+(-1)^p ad(b),\forall a\in A^p, b\in A$$. A dg right $A$-module is a
$\mathbb{Z}$-graded right $A$-module
$M=\displaystyle\oplus_{i\in\mathbb{Z}}M^i $ with a graded $k$-linear
differential $d: M^i\ra M^{i+1}$ such that
$$
d(ma)=d(m)a+(-1)^nmd(a), \forall m\in M^n, a\in A.
$$
It is useful to look at each dg $A$-module as a complex
$$\cdots\ra M^{i-1}\stackrel{d^{i-1}}\ra
M^i\stackrel{d^i}\ra M^{i+1}\ra\cdots$$ of $k$-modules with some extra properties.

A morphism between right dg $A$-modules $f: M\ra N$ is a morphism of the
underlying graded $A$-modules which is homogeneous of degree zero
and commutes with the differential.
Denote by $\rGr{A}$ the category of graded right $A$-modules with graded $A$-module homomorphisms.
Let $M$ and $N$ be right dg $A$-modules. Set $\HOM_A(M,N)^n=\Hom_{\rGr{A}}(M,N<n>)$ for each
$n\in\mathbb{Z}$, it consists of $k$-linear maps $f: M\lra N$ which are homogeneous of degree $n$ and satisfy
$f(ma) =f(m)a$ for all homogeneous elements $a\in A$. The graded vector space
$$\HOM_A(M,N)=\oplus_{i\in\mathbb{Z}}\HOM_A(M,N)^i=\oplus_{i\in\mathbb{Z}}\Hom_{\rGr{A}}(M,N<i>)$$
has a natural differential $d$ such that $d(f)=d_N\circ f+ (-1)^{|f|+1} f\circ d_M.$
Furthermore,
$\HOM_A(M,M)$ becomes a dg algebra with this differential and the usual composition as multiplication.
Denote by $\mathcal {C}_{dg}{A}$ the category of differential graded right $A$-modules with morphism space $\mathcal {C}_{dg}{A}(M,N)=\HOM_A(M,N)$.

\subsection{Derived category of dg module categories}
Let $A$ be a dg $k$-algebra. Let $\mathscr{C}(A)$ be the category of dg $A$-modules.
 A morphism $f: X\lra Y$ in $\mathscr{C}(A)$ is a morphism in $\rGr{A}$ which is a chain map of complexes of
k-modules. We have
$$\Hom_{\mathscr{C}(A)}(X, Y)=\mathcal {Z}^0(\HOM_A(X,Y))=\{f\in\HOM_A(M,N)^0| d(f)=0\}.$$
Note that $\mathscr{C}(A)$ is an abelian category and comes with
a canonical shifting $[1]:\mathscr{C}(A)\lra \mathscr{C}(A)$
which comes from the canonical shifting of $\rGr{A}$, by defining $d^n_{M[1]}=d_M^{n+1}$
 for each $n\in\mathbb{Z}$. Then
we have an obvious faithful forgetful functor
$$F: \mathscr{C}(A)\lra \rGr{A}$$
which is also dense on objects since we can interpret
each graded $A$-module as an object of $\mathscr{C}(A)$ with zero differential. Viewing the objects
of $\mathscr{C}(A)$ as complexes of $k$-modules, we clearly have, for each $p\in\mathbb{Z}$, the $p$-th homology
functor
$$H^p : \mathscr{C}(A)\lra \rModcat{k}.$$
A morphism $f : X\lra Y$ in $\mathscr{C}(A)$ is called a quasi-isomorphism if $H^p(f)$ is an isomorphism, for all $p\in\mathbb{Z}$. A dg A-module $X$ is called
acyclic if $H^p(X)=0$, for all $p\in\mathbb{Z}$.

For any dg $A$-module $X=\oplus X^i$, we define a functor $F_{\rho}:\rGr{A}\lra \mathscr{C}(A)$ given by
 $$F_{\rho}(X)=X\oplus X<1>$$
 as graded $A$-modules.
For a map $f\in\Hom_{\rGr{A}}(X,Y)$, $F_{\rho}(f):F_{\rho}(X)\lra F_{\rho}(Y)$ given by
$$F_{\rho}(\left[\begin{smallmatrix}x\\y \end{smallmatrix}\right])=\left[\begin{smallmatrix}f(x)\\f(y) \end{smallmatrix}\right].$$
Then we get a dg $A$-module
$F_{\rho}(X)=X\oplus X<1>$ given by
$$\left[\begin{smallmatrix}x\\y \end{smallmatrix}\right]a=\left[\begin{smallmatrix}xa\\ya+(-1)^{|x|}xd(a) \end{smallmatrix}\right]$$
with the differential
$d=\left[\begin{smallmatrix}0&1\\0&0 \end{smallmatrix}\right]$. It is well-known that $(F,F_{\rho})$ is an adjoint pair.
Note that $F_{\rho}F(X)$ is a projective-injective object in $\mathscr{C}(A)$, and we have a canonical exact sequence
$0\lra X\lra F_{\rho}F(X)\lra X[1]\lra 0$
which splits in $\rGr{A}$ but not in $\mathscr{C}(A)$.
The exact structure on $\mathscr{C}(A)$ is the exact sequence
$$0\lra X\lra Y\lra Z\lra 0$$
in $\mathscr{C}(A)$
which is split in $\rGr{A}$. With this exact structure, $\mathscr{C}(A)$ is a Frobenius exact category.
Then the stable category $\mathscr{C}(A)$ with respect to the given exact structure is denoted by $\K{A}$, and is called the homotopy category of $A$.
It is a triangulated category, and each triangle in $\K{A}$ is isomorphic to
$$X\lraf{f}Y\lraf{\left[\begin{smallmatrix}-1\\0 \end{smallmatrix}\right]}X[1]\oplus Y\lraf{\left[\begin{smallmatrix}0 &1 \end{smallmatrix}\right]}X[1]$$
for some $f:X\lra Y$. Let $f:M\ra N$ be a morphism of dg $A$-modules. We say that $f$ is a
null-homotopic if we get
$$f=dr+rd,$$
where $r:M\ra N$ is a morphism
of the underlying graded $A$-modules which is homogeneous of degree
$-1$. Denote by $\Sigma$ the class of all homotopy class of
quasi-isomorphisms. Then $\Sigma$ is a multiplicative system in $\K{A}$ compatible with the triangulation and the localization
$$\D{A}:=\K{A}[\Sigma^{-1}]$$
is called the derived category of $A$. It is also a triangulated category with translation functor $[1]$ induced from $\mathscr{C}(A)$.

For two dg $A$-modules $X$ and $Y$, there is an isomorphism
$$
\Hom_{\K{A}}(X,Y[n])\simeq H^n(\HOM_A(X,Y))
$$
for each $n\in\mathbb{Z}$. We observe an isomorphism $\HOM_A(A,X)\simeq X$ of complexes sending $f$ to $f(1)$.  Then the above isomorphism induces the following isomorphism
$$
\Hom_{\K{A}}(A,X[n])\simeq H^n(X).
$$

It is well-known that $\D{A}$ can be realized as the Verdier quotient category of $\K{A}$ by its full subcategory of acyclic dg $A$-modules.
A dg $A$-module $P$ is called homotopically projective if $\HOM_A(P,M)=0$, for every
acyclic dg $A$-module $M$, which is equivalent to the condition that the canonical functor $\K{A}\lra\D{A}$ induces an isomorphism $$\Hom_{\K{A}}(P,X)\simeq\Hom_{\D{A}}(P,X)$$
for any dg $A$-module $X$.

For a triangulated category $\mathcal {T}$ and a class $\mathcal {S}$ of objects, we denote by $\thick_{\mathcal {T}}{\mathcal {S}}$ the smallest
thick subcategory of $\mathcal {T}$ containing $\mathcal {S}$. Denote by $\Tria_{\mathcal {T}}{\mathcal {S}}$ is the smallest subcategory of
$\mathcal {T}$ containing $\mathcal {S}$ and closed under coproducts. Recall that a thick subcategory is a triangulated subcategory which is closed under direct
summands. We say an object $M$ in $\mathcal {T}$ is compact if the functor $\Hom_{\mathcal {T}}(M,-)$ commutes with arbitrary coproducts.

The following easy lemma is useful in the later, for the proof we refer to \cite[Corollary 3.19]{Nicolas2012a}.

\begin{Lem}\label{4.2}  Let $F: \mathcal {T} \lra \mathcal {T}'$ be a triangle functor of triangulated categories and
$T\in\mathcal {T}$ an object. Then
$$F(\thick_{\mathcal {T}}(T))\subseteq\thick_{\mathcal {T}'}(F(T)).$$ \label{thick}
\end{Lem}

The following is a very well-known fact \cite[theorem 5.3]{Keller1994}.

\begin{Prop} Let $A$ be a dg algebra. The compact objects of $\D{A}$ are precisely
the objects of $\thick_{\D{A}}(A):=\per{A}$.
\end{Prop}

The triangulated subcategory $\per{A}$ is called the perfect derived category of $A$.

\subsection{Differential graded bimodules}
We give in this subsection some useful lemmas for later use.

Let $A$ and $B$ be two dg algebras. A dg $A$-$B$-bimodule $X$ is a left dg $A$-module as well as right dg $B$-module such that
$$d(amb)=(da)mb+(-1)^p a(dm)b+(-1)^{p+q}am(db),$$ for all $a\in A^p, m\in M$ and $b\in B^q$.
And the canonical map $A\lra\HOM_B(X, X),$
sending $a$ to $l_a$ with $l_a(x)=ax$ is a homomorphism of dg algebras.
Similarly, the canonical map $B\lra\HOM_{A^{op}}(X, X)^{op},$ sending $b$ to $r_b$ with $r_b(x)=(-1)^{|b||x|} xb$ is a homomorphism of dg algebras.

Recall that a dg $A$-$B$-bimodule $X$ is called left quasi-balanced if the canonical map $A\lra\HOM_B(X, X)$ of dg algebras is a quasi-isomorphism.
Dually, $X$ is called right quasi-balanced if the canonical map $B\lra\HOM_{A^{op}}(X, X)$ of dg algebras is a quasi-isomorphism. If both the canonical map $A\lra\HOM_B(X, X)$ and the canonical map $B\lra\HOM_{A^{op}}(X, X)$ of dg algebras are  quasi-isomorphisms, then $X$ is called a quasi-balanced bimodule.
Let $X$ be a dg $A$-$B$-bimodule. Then $\HOM_B(_AX_B, Y_B)$ becomes a right dg $A$-module by $(fa)(x)=f(ax)$ for $f\in\HOM_B(_AX_B, Y_B)$ and $x\in X$. Similarly, $\HOM(_AX_B, _AY)$ is a left dg $B$-module by
$(bf)(x)=(-)^{|x|(|f|+|b|)}f(xb)$ for $f\in\HOM_B(_AX_B, _AY)$ and $x\in X$.

Here we give a condition for a fully faithful exact functor between triangulated
categories to be an equivalence. We also refer the reader to \cite[Theorem 3.3]{Bridgeland1999}.

\begin{Lem}\label{ffd} Let $\mathcal {T}$ and $\mathcal {S}$ be triangulated categories.
 Let $F:\mathcal {T} \lra \mathcal {S}$ and $G:\mathcal {S} \lra \mathcal {T}$
be an adjoint pair of triangle functors. Then we have the following

$(1)$ If the unit $\eta: Id_{\mathcal {T}}\lra GF$ is an isomorphism and $GX\simeq 0\Leftrightarrow X\simeq 0$, then $F$ and $G$ induce mutually inverse equivalences between $\mathcal {T}$ and $\mathcal {S}$.

$(2)$ If the counit $\varepsilon: FG\lra Id_{\mathcal {S}}$ is an isomorphism and $FY\simeq 0\Leftrightarrow Y\simeq 0$, then $F$ and $G$ induce mutually inverse equivalences between $\mathcal {T}$ and $\mathcal {S}$.

\end{Lem}
\begin{proof}
(1) For any $Y\in \mathcal {S}$, there is a triangle $$ FGY\lraf{\varepsilon_{Y}} Y\lra Z\lra (FGY)[1]$$
Applying the functor $G$, then we get the following triangle
$$\xymatrix{
  GFGY \ar[r]^{G(\varepsilon_{Y})} &   GY\ar[r] & GZ \ar[r] & (GFGY)[1].    \\
  GY \ar[u]^{\eta_{GY}}\ar@{=}[ur].                     }
$$
Therefore, we get $1_{GX}=G(\varepsilon_{Y}) \eta_{GY}$.
Since $\eta_{GY}: GY\lra GFGY$ is an isomorphism, $GFGY\lraf{G(\varepsilon_{Y})} GY$ is an isomorphism. Then
$GZ\simeq 0$. By assumptation, $Z\simeq 0$. Consequently,  $FGY\simeq Y$.

 $(2)$ The proof of $(2)$ is similar to that of $(1)$.
\end{proof}

The following lemma is the Morita theory for derived categories of
dg sense which was proved by Keller \cite{Keller1994}. We give an another characterization by Lemma \ref{ffd}.
\begin{Lem}\label{Morita-lemma}
Let $_AM_B$ be a dg $A$-$B$-bimodule such that $M_B$ is a homotopically projective compact generator in $\D{B}$ and
the canonical map
$$A\lra \HOM_B(M,M), \quad (a\mapsto (m\mapsto am))$$
is a quasi-isomorphism of dg algebras.
Then
$-\otimesL_AM$ and $\rHom_B(M,-)$ induces mutually inverse equivalences between $\D{A}$ and $\D{B}$.
\end{Lem}

\begin{proof} Let $F=-\otimes_AM$ and $G= \HOM_B(M,-)$. Then we consider the left and right derived functor $\bar{F}:=-\otimesL_AM$ and $\bar{G}:=\rHom_B(M,-)$, respectively.
We know that $G$ preserves a quasi-isomorphism. Then we have the following
$$A_A\lraf{\eta_A}GFA=\HOM_B(_AM_B,A\otimes_AM)\lra \HOM_B(M,M), \quad (a\mapsto (m\mapsto a\otimes m\mapsto am))$$
Hence $\eta_A$ is a quasi-isomorphism. Since $M_B$ is a homotopically projective, the unit of $A\lra \bar{G}\bar{F}A$ is an isomorphism. Let
$$\mathscr{X}=\{X\in\D{A}|\;\eta_X: X\lra \bar{G}\bar{F}X\; \mbox{is an isomorphism}\}.$$ Then $\mathscr{X}$ is a subcategory of $\D{A}$ containing $A_A$.
Let $\K{A}_{ac}=\{M\in\K{A}|\; M \;\mbox{is acyclic as $k$-complex} \}$ and $\K{A}_{ac}^{\perp}=\{X\in\K{A}|\; \Hom_{\K{A}}(\K{A}_{ac},X)=0 \}$. Then we have the following triangle functors
$$\Tria(A)\hraf{\lambda}\K{A}\lraf{q}\D{A}, \K{A}_{ac}^{\perp}\hraf{\mu}\K{A}\lraf{q}\D{A}.$$ Here $q\lambda$ and $q\mu$ are equivalent of triangle functors. Let $\gamma$ and $\gamma'$ be the quasi-inverse of $q\lambda$ and $q\mu$, respectively. Set $p:=\lambda\gamma$ and $i:=\mu\gamma'$. We thus have the following diagram
$$\xymatrix{
\K{A}\ar[dd]|q \ar@/^/[dd]^p\ar@/ ^/[rr]^F&& \K{B}\ar@/^/[ll]^G\ar[dd]|q \ar@/^/[dd]^p\\
&&\\\D{A} \ar@/^/[uu]^i\ar@/^/[rr]^{LF}&& \D{B}\ar@/^/[ll]^{LG}\ar@/^/[uu]^i.
}$$
Note that $\bar{F}=qFp$ and $\bar{F}=qGi$. The functors $p,F,q,G$ preserve coproducts.  It follows that $i(\coprod X)\simeq i(\coprod qi(X))\simeq i(q\coprod i(X))\simeq iq(\coprod i(X))\simeq\coprod i(X)$.
Therefore, $\mathscr{X}$ is closed under coproducts. Consequently, we get $\Tria_{\D{A}}(A)\subseteq\mathscr{X}\subseteq\D{A}$. Then $\mathscr{X}=\D{A}$.
This implies that $\eta_X: X\lra \bar{G}\bar{F}(X))$ is an isomorphism. For any $X\in\D{B}$ such that $\bar{G}(X)=0$. Thus $qGi(X)=0$. Therefore, $q\HOM_B(M_B,i(X))=0$.
Thus $\HOM_B(M_B,i(X))$ is acyclic. It follows that $\Hom_{\K{B}}(M_B,i(X)[n])=0$ for $n\in \mathbb{Z}$. So, $\Hom_{\D{B}}(M_B,X[n])=0$. Then $X\simeq0$ by $M_B$ is a compact generator in $\D{B}$. This completes the proof by Lemma \ref{ffd}.
\end{proof}

\begin{Koro}\cite{Keller1994}\label{quasi-iso} Suppose that $\phi: A\lra B$ is a dg algebra homomorphism. If $\phi$ is a quasi-isomorphism of dg algebras, then
$-\otimesL_AB$ and $\rHom_B(B,-)$ induces mutually inverse equivalences between $\D{A}$ and $\D{B}$.
\end{Koro}
\begin{proof} This is the special case of Lemma \ref{Morita-lemma} for $M=_AB_B$, $_AB_B$ is a dg $A$-$B$-bimodule such that $B_B$ is a homotopically projective compact generator in $\D{B}$.
\end{proof}

\section{Derived equivalences of dg-algebras induced from triangles}\label{triangles induced derived euivalences}
Throughout this section, we fix $A$ a dg algebra and denote its homotopy category
by $\K{A}$. The goal of this section is to get derived equivalences of two dg algebras which are obtained from some triangles in $\K{A}$.
 As an application, we give an affirmative answer to Dugas' question \cite{Dugas2015}. First we need the following well-known observation.

\begin{Lem}\cite[Lemma 8.4]{Iyama2013b}\label{Ia} Let $F: \mathcal {T} \lra \mathcal {T}'$ be a triangle functor of triangulated categories and
$T$  an object in $\mathcal {T}$. If
 $$F_{T,T}: \Hom_{\mathcal {T}}(T,T[n]) \lra \Hom_{\mathcal {T}'}(F(T), F(T)[n])$$
  is an isomorphism for any $n\in\mathbb{Z}$, then
  $$F : \thick_{\mathcal {T}}(T)\lra \mathcal {T}'$$
  is fully faithful. \label{fully-faithful}
\end{Lem}

We generalizes \cite[Lemma 8.5]{Iyama2013b} to dg algebra case in the following.
\begin{Lem}\label{3.4} Let $X$ be a dg $A$-module and $\Lambda=\HOM_A(X,X)$. The we have a triangle equivalence between $\thick_{\K{A}}(X)$ and $\per{\Lambda}$.
\end{Lem}

\begin{proof} Consider the following functor
$$\HOM_A(X,-): \C{A}\lra \C{\Lambda}$$
 which sends $Y$ to a dg $\Lambda$-module $\HOM_A(X,Y)$. A map $f: Y_1\lra Y_2$ in $\C{A}$ yields
  $$\HOM_A(X,f): \HOM_A(X,Y_1)\lra\HOM_A(X,Y_2).$$
Then
this functor sends a null-homotopic morphism of dg $A$-modules to a null-homotopic morphism
of dg $\Lambda$-modules. Therefore, $\HOM_A(X,-)$ induces a triangle functor
$\K{A}\lra\K{\Lambda}$. By composing with the canonical functor
$\K{\Lambda}\lra \D{\Lambda}$, we have a triangle functor $\K{A}\lra\D{\Lambda}$
which sends $X$ to $\Lambda$. We thus get a triangle functor from $\thick_{\K{A}}(X)$ to $\thick_{\D{\Lambda}}(\Lambda)=\per{\Lambda}$ by Lemma \ref{Ia}.
Then we have the following commutative diagram
$$\xymatrix{
\Hom_{\K{A}}(X,X[n])\ar[d]^{\simeq}\ar[r]^{F_{X,X[n]}} & \Hom_{\D{\Lambda}}(\Lambda,\Lambda[n])\ar[d]^{\simeq}\\
H^n(\Lambda) \ar[r]^{\simeq}& H^n(\Lambda), \\
}$$
 for any $n\in\mathbb{Z}$. It completes the proof by Lemma \ref{fully-faithful}.
\end{proof}

\begin{Lem}\label{3.5} Suppose $X$ and $Y$ are dg $A$-modules such that $Y\in\thick_{\K{A}}(X)$. Let $\Lambda=\HOM_A(X,X)$. Then we have the following quasi-isomorphism:
$$\begin{array}{rl}
\HOM_A(Y,Z)\lra \HOM_{\Lambda}(\HOM_A(X,Y),\HOM_A(X,Z))\quad f\mapsto (g\mapsto f\circ g),\end{array}$$
for any dg $A$-module $Z.$
\end{Lem}

\begin{proof} As in the proof in Lemma \ref{3.4}, the functor $\HOM_A(X,-)$ induces a triangle functor $\K{A}\lra\K{\Lambda}$. Applying the functor $H^i$ to
$$\HOM_A(Y,Z)\lra \HOM_{\Lambda}(\HOM_A(X,Y),\HOM_A(X,Z)),$$
it yields that
$$\Hom_{\K{A}}(Y,Z[i])\lra \Hom_{\K{\Lambda}}(\HOM_A(X,Y),\HOM_A(X,Z)[i]).$$
We now let
$$\mathscr{X}=\{Y\in\K{A}|\;\mbox{for any object}\; Z\in\K{A}, \mbox{the functor}\; \HOM_A(X,-) \;\mbox{induces the following isomorphism:}$$
$$Hom_{\K{A}}(Y,Z[i])\lra \Hom_{\K{\Lambda}}(\HOM_A(X,Y),\HOM_A(X,Z)[i])\}.$$
Then $\mathscr{X}$ is a subcategory of $\K{A}$. Clearly, $\mathscr{X}$ is closed under
shifts and direct summands. Suppose that
$$Y_1\lra Y_2\lra Y_3\lra Y_1[1]$$
is a triangle in $\K{A}$ with $Y_1, Y_2\in \mathscr{X}$. Then we have the following commutative diagram
$$\xymatrix@M=2mm@C=2mm{
\K{A}(Y_1[1],Z[i])\ar^{\simeq}[d] \ar[r]
                &\K{A}(Y_3,Z[i]) \ar[d]\ar[r] &\K{A}(Y_2,Z[i]) \ar^{\simeq}[d]\ar[r] &\K{A}(Y_1,Z[i]) \ar^{\simeq}[d]\\
\K{\Lambda}((X,Y_1[1]),(X,Z)[i]) \ar[r]&\K{\Lambda}((X,Y_3),(X,Z)[i])\ar[r]&\K{\Lambda}((X,Y_2),(X,Z)[i])\ar[r]& \K{\Lambda}((X,Y_1),(X,Z)[i]) .}
$$ It follows that $Y_3\in\mathscr{X}$. Therefore, $\mathscr{X}$ is a thick subcategory of $\K{A}$. If $Y=X$, then
$$\begin{array}{rl}
\Hom_{\K{\Lambda}}(\HOM_A(X,X),\HOM_A(X,Z)[i])\simeq\Hom_{\K{\Lambda}}(\Lambda,\HOM_A(X,Z)[i])\simeq\Hom_{\D{\Lambda}}(\Lambda,\HOM_A(X,Z)[i])\\\simeq H^i(\HOM_A(X,Z))
 \simeq\Hom_{\K{A}}(X,Z[i]).
\end{array}$$
Thus $X\in\mathscr{X}$. Consequently, $\thick_{\K{A}}(X)\subseteq\mathscr{X}$. This completes the proof.
\end{proof}

The following lemma is useful, for convenience we give the proof here. We also thank Professor Bernhard Keller for his discussions in the proof of the following lemma.

\begin{Lem}\label{3.6} Suppose that $X$ and $Y$ are dg $A$-modules such that $Y\in\thick_{\K{A}}(X)$ and  $X\in\thick_{\K{A}}(Y)$. Let $\Lambda=\HOM_A(X,X)$ and $\Gamma=\HOM_A(Y,Y)$. Then there is a triangle equivalence
between $\D{\Lambda}$ and $\D{\Gamma}$.
\end{Lem}

\begin{proof} Consider the dg $\Gamma$-$\Lambda$-bimodule $\HOM_A(X,Y)$. By Lemma \ref{3.5}, we have the following quasi-isomorphism:
$$\begin{array}{rl}
\HOM_A(Y,Y)\lra \HOM_{\Lambda}(\HOM_A(X,Y),\HOM_A(X,Y))\quad f\mapsto (g\mapsto f\circ g).\end{array}$$
Similarly to the proof of Lemma \ref{3.5}, it follows that
$$\begin{array}{rl}
\HOM_A(X,X)\lra \HOM_{\Gamma^{op}}(\HOM_A(X,Y),\HOM_A(X,Y))\quad f\mapsto (g\mapsto (-1)^{|f||g|}g\circ f).
\end{array}$$
 is a quasi-isomorphism of dg algebras. Then the dg $\Gamma$-$\Lambda$-bimodule $\HOM_A(X,Y)$ is quasi-balanced.

It follows from Lemma \ref{3.4} that $\HOM_A(X,Y)_{\Lambda}\in\per{\Lambda}$
and $_{\Gamma}\HOM_A(X,Y)\in\per{\Gamma^{op}},$
since $Y\in\thick_{\K{A}}(X)$ and  $X\in\thick_{\K{A}}(Y)$. So $\Lambda\in \thick_{\D{\Lambda}}({\HOM_A(X,Y)_{\Lambda}})$.
Then $\thick_{\D{\Lambda}}({\HOM_A(X,Y)_{\Lambda}})=\per{\Lambda}.$
Consequently, $\HOM_A(X,Y)$ is a compact generator for $\D{\Lambda}$. By Lemma \ref{Morita-lemma},
we have a triangle equivalence between $\D{\Lambda}$ and $\D{\Gamma}$
defined by the dg $\Gamma$-$\Lambda$-bimodule $\HOM_A(X,Y)$.
\end{proof}

\begin{Koro}\cite[Proposition 8.3]{Efimov2009} or \cite[Proposition 3.3]{Porta2014}.\label{3.9} Assume that dg $A$-modules $X$ and $Y$ are homotopy equivalent. Then the dg algebras $\Lambda=\HOM_A(X,X)$, $\Gamma=\HOM(Y,Y)$ are quasi-isomorphic. Moreover, we have a triangle equivalence
$\D{\Lambda}\simeq \D{\Gamma}$.
\end{Koro}

\begin{Remark}\label{5.5} Let $A$ be a finite dimensional $k$-algebra, and $X$ be a finitely generated $A$-module.
Let $\cpx{P}_X$ be a projective resolution of $X$.
Then we have $\RHom_A(X,X)=\cpx{\Hom_A}(\cpx{P}_X,\cpx{P}_X),$ where $\RHom_A(X,-)$ is the right derived functor.
Suppose $\cpx{P_X}$ and $\cpx{Q_X}$ are two projective resolutions of $X$ which are homotopy equivalent,
and let $\Lambda_1=\cpx{\Hom_A}(\cpx{P_X},\cpx{P_X})$ and $\Lambda_2=\cpx{\Hom_A}(\cpx{Q_X},\cpx{Q_X})$
be two  dg algebras.  Then we have a triangle equivalence between $\D{\Lambda_1}$ and $\D{\Lambda_2}$
by Lemma \ref{3.9}.

\end{Remark}

\begin{proof} We give a new proof of
$\D{\Lambda}\simeq \D{\Gamma}$ by Lemma \ref{3.6} for reader's
convenience. Since dg $A$-modules $X$ and $Y$ are homotopy equivalent,
$$\thick_{\K{A}}(X)=\thick_{\K{A}}(Y).$$ Let $\Lambda=\HOM_A(X,X)$ and $\Gamma=\HOM_A(Y,Y)$. There is a dg $\Gamma$-$\Lambda$-bimodule $\HOM_A(X,Y)$ which yields a triangle equivalence between $\D{\Lambda}$
and $\D{\Gamma}$ by Lemma \ref{3.6}.
\end{proof}

The following is the main result in this section.

\begin{Theo}\label{4.1} Let $M$ be a dg $A$-module in homotopy category $\K{A}$ of dg algebra $A$ and
$$X\lraf{f} M_1\lraf{g} Y\lraf{h} X[1]$$
be a triangle in $\K{A}$ with $M_1\in <M>$, where $<M>$ denotes the full additive subcategory generated by $\cup_{i\in\mathbb{Z}}\add M[i]$. Then there is
a derived equivalence of dg algebras $\HOM(X\oplus M, X\oplus M)$ and $\HOM(Y\oplus M, Y\oplus M)$.
\end{Theo}

\begin{proof} Suppose that $M$ is a dg $A$-module in $\K{A}$.
Let
$$X\lraf{f} M_1\lraf{g} Y\lraf{h} X[1]$$
be a triangle in $\K{A}$ with $M_1\in <M>$. Then there is a triangle
$$
X\lraf{f} M_1\oplus M\lraf{g} Y\oplus M\lraf{h} X[1].
$$
Therefore, $Y\oplus M\in\thick(X\oplus M)$ and $X\oplus M\in\thick(Y\oplus M)$. Set $U=X\oplus M$ and $V=Y\oplus M$. By Lemma \ref{3.6}, the dg bimodule $\HOM_A(U,V)$ defines a derived equivalence
between $\HOM(X\oplus M, X\oplus M)$ and $\HOM(Y\oplus M, Y\oplus M)$.
\end{proof}

By Theorem \ref{4.1} and \cite [Theorem 4.1]{Dugas2015}, we can restate the above theorem for the case of a $k$-algebra $A.$

\begin{Koro}\label{4.2} Suppose that $A$ is a $k$-algebra and $\cpx{M}$ is a complex in $\K{A}$.
Let
$$\cpx{X}\lraf{\cpx{f}} \cpx{M_1}\lraf{\cpx{g}} \cpx{Y}\lraf{\cpx{h}} \cpx{X}[1]$$
 be a triangle in $\K{A}$ with $\cpx{M_1}\in <\cpx{M}>$.  Assume that $\cpx{f}$ is a left $<\cpx{M}>$-approximation of $\cpx{X}$ and $\cpx{g}$ is a right  $<\cpx{M}>$-approximation of $\cpx{Y}$. Then

\smallskip
$(1)$. There is a derived equivalence between dg algebras $\cpx{\Hom_A}(\cpx{X}\oplus \cpx{M}, \cpx{X}\oplus \cpx{M})$ and $\cpx{\Hom_A}(\cpx{Y}\oplus \cpx{M}, \cpx{Y}\oplus \cpx{M})$.

$(2)$. The cohomology rings $H^{\mathbb{Z}}(\cpx{X}\oplus \cpx{M})$  and $H^{\mathbb{Z}}(\cpx{Y}\oplus \cpx{M})$ are also derived equivalent, where
$H^{\mathbb{Z}}(\cpx{X}\oplus \cpx{M})=\bigoplus_{i\in\mathbb{Z}}\Hom_{\K{A}}(\cpx{X}\oplus \cpx{M}, (\cpx{X}\oplus \cpx{M})[i])$.
\end{Koro}

\begin{Remark}
There is a derived equivalence between dg algebras $\cpx{\Hom_A}(\cpx{X}\oplus \cpx{M}, \cpx{X}\oplus \cpx{M})$ and $\cpx{\Hom_A}(\cpx{Y}\oplus \cpx{M}, \cpx{Y}\oplus \cpx{M})$.
 Under some mild conditions, we get the cohomology rings of these dg algebra are derived equivalent. This gives
an affirmative answer to a problem of Dugas \cite{Dugas2015} in some special case.

\end{Remark}

\begin{Koro}\cite[Theorem 5.2]{Dugas2015} Let $A$ be a symmetric algebra and $\cpx{X}$ and $\cpx{M}$ be any complexes in $\Kb{\rpmodcat{A}}$. Then there exists a left
$<\cpx{M}>$-approximation $\cpx{X}\lraf{\cpx{f}}\cpx{M'}$ of $\cpx{X}$ in $\Kb{A}$. If $\cpx{Y}=\cone(\cpx{f})$, we get the derived equivalence between
$\cpx{\Hom_A}(X\oplus M, X\oplus M)$ and $\cpx{\Hom_A}(Y\oplus M, Y\oplus M)$.
\end{Koro}

\begin{proof} Let $\widetilde{\K{A}}$ denote the orbit category $\K{A}/[1]$. Then there is a triangle
$$\cpx{X}\lraf{\cpx{f}}\cpx{M'}\lraf{\cpx{g}} \cpx{Y}\lraf{h} \cpx{X}[1]$$
 such that $\cpx{M'}\in\add_{\widetilde{\K{A}}} M$, $\cpx{f}$ is a left $<\cpx{M}>$-approximation and $\cpx{g}$ is a right $<\cpx{M}>$-approximation by \cite[Theorem 4.1]{Dugas2015}.
It completes the proof by Theorem \ref{4.1}.
\end{proof}

\begin{exm}
  \cite[Section 7]{Dugas2015} We now show how to realize dg algebras in one concrete case. Set
$$A=k[x,y]/(x^n-y^s,xy)$$
and consider
$$\cpx{T}=A\lraf{(x,0)}A\oplus A$$
 concentrated in degree $-1$ and $0$.
Write  $\cpx{T}=\cpx{T}_1\oplus\cpx{T}_2$, where
$$\cpx{T}_1:=(0\lraf{0}A)\mbox{~~~and~~~} \cpx{T}_2:=(A\lraf{x}A). $$
Then we conclude that
 $$\cpx{T}_2\lraf{\left[\begin{smallmatrix}\gamma\\\gamma\epsilon \end{smallmatrix}\right]}\cpx{T}_1[1]\oplus \cpx{T}_1$$
  is a left
$<\cpx{T}_1>$-approximation, yielding the following triangle in $\Kb{\rpmodcat{A}}$
$$
(A\lraf{x}A) \lraf{(1,y)} (A\lraf{0}A)\lraf{(x,1)}(A\lraf{y}A)\lra(A\lraf{x}A)[1],
$$
where $$ \xymatrix@M=1mm{
A\ar[d]_{\gamma:\quad}^1 \ar[r]^x & A \ar^{,}[d]\\
A \ar[r] & 0
}\xymatrix@M=1mm{
A\ar_{\epsilon: \quad}[d] \ar[r]^x & A\ar[d]^y\ar[r] & 0\ar^{.}[d]\\
0\ar[r]& \ar[r]A \ar[r]^x & A
}$$
Theorem \ref{4.1} thus shows that two dg algebras
$$\left[\begin{smallmatrix}\xymatrix@M=1mm{ A_0} & \xymatrix@M=1mm{ A_0\ar[r]^{-x} & A_1}\\\xymatrix@M=1mm{ A_{-1}\ar[d]_{x}\\ A_0}  &
\xymatrix@M=1mm{
A_{-1}\ar[d]^x \ar[r]^{x} & A_0\ar[d]_{x}\\
A_0 \ar[r]^{-x} & A_1
}\end{smallmatrix}\right], \left[\begin{smallmatrix}\xymatrix@M=1mm{ A_0} & \xymatrix@M=1mm{ A_0\ar[r]^{-y} & A_1}\\\xymatrix@M=1mm{ A_{-1}\ar[d]_{y}\\ A_0}  &
\xymatrix@M=1mm{
A_{-1}\ar[d]^y \ar[r]^{y} & A_0\ar[d]_{y}\\
A_0 \ar[r]^{-y} & A_1
}\end{smallmatrix}\right]$$ are derived equivalent, where $A_{-1}=A_0=A_1=A$ concentrate in degree $-1,0,1$, respectively.
\end{exm}

\section{Derived equivalence of dg-algebras induced from standard derived equivalences of finite dimensional algebras}\label{section 5}
Starting from a standard derived equivalence of finite dimensional algebras, we construct in this section
derived equivalences of two dg algebras. Moreover, under some conditions, the cohomology rings of these dg
algebra are also derived equivalent.

Throughout this section, suppose $k$ is a field. Let $A$ and $B$ be finite dimensional $k$-algebras.
Recall that a
standard derived equivalence between derived categories
$\Df{\rModcat{A}}$ and $\Df{\rModcat{B}}$ is an exact functor such that it is an
equivalence and is isomorphic to $\RHom(\cpx{X},-)$ for some object
$\cpx{X}$ of $\Db{\rModcat{A\otimes B}}$ \cite{Rickard1991}. An object $\cpx{X}$
of $\Db{\rModcat{A\otimes B}}$ is called a two-sided tilting complex if
it induces such an equivalence.

\begin{Lem} \cite[Lemma 2.1]{Hu2010}
Let $F: \D{\rModcat{A}}\lra \D{\rModcat{B}}$ be a derived equivalence.
 For each $A$-module $X$, the image $F(X)$ is isomorphic, in $\D{\rModcat{B}}$, to a complex $\cpx{\bar{T}_X}$ of the form
 $$0\lra \bar{T}_X^0\lra \bar{T}_X^1\lra\cdots\lra \bar{T}_X^n\lra 0$$
 with $\bar{T}_X^i$ projective for all $1\leq i\leq n$. Moreover,  if $X$ admits a projective resolution $\cpx{P}_X$ with $P_X^i$ finitely generated for $0\leq i\leq m$ with $m\geq n$. Then $\bar{T}_X^i$ can be chosen to be finitely generated for all $0\leq i\leq n$ and $\bar{T}_X^0$ admits a projective resolution $\cpx{Q}$
where $Q^i$ is finitely generated for $-m\leq i\leq 0.$
\label{lemma-modCompForm}
\end{Lem}

\begin{Lem} \cite[Proposition 3.4]{Hu2010}
Let $F:\Db{\rmodcat{A}}\lra \Db{\rmodcat{B}}$ be a derived equivalence. Then there is an
additive functor
$$\underline{F}: \rstmodcat{A}\ra\rstmodcat{B}$$ sending $X$ to
$\underline{F}(X)$, such that the following diagram
$$\xymatrix{
      \rstmodcat{A}\ar[r]^(.35){{\rm can}}\ar[d]_{\underline{F}} &
      \Db{\rmodcat{A}}/\Kb{\rpmodcat{A}}\ar[d]^{{F}}\\
      \rstmodcat{B}\ar[r]^(.35){{\rm can}} & \Db{\rmodcat{A}}/\Kb{\rpmodcat{B}}
    }$$
   is commutative up to natural isomorphism.\label{2.3}
\end{Lem}

The following theorem is the main result of this section.

\begin{Theo}\label{5.2} Suppose $k$ is a field. Let $F: \D{\rModcat{A}}\lra \D{\rModcat{B}}$ be a standard derived equivalence and
$$\underline{F}: \rstmodcat{A}\lra\rstmodcat{B}$$
be the additive functor induced  by $F$ as in Lemma \ref{2.3}.
Then there is a derived equivalence between dg algebras $\RHom_{A}(A\oplus X,A\oplus X)$ and $\RHom_{B}(B\oplus\underline{F}(X),B\oplus\underline{F}(X))$ for each
$A$-module $X$.
\end{Theo}

\begin{Remark} If $X\in {}^{\perp}A$, then there is a derived equivalence between dg algebras
$\RHom_{A}(A\oplus X,A\oplus X)$ and $\RHom_{B}(B\oplus\underline{F}(X),B\oplus\underline{F}(X))$ by Theorem \ref{5.2}.
Moreover, the cohomology rings $H^{\mathbb{Z}}(A\oplus X)$  and $H^{\mathbb{Z}}(B\oplus\underline{F}(X))$ of $\RHom_{A}(A\oplus X,A\oplus X)$ and $\RHom_{B}(B\oplus\underline{F}(X),B\oplus\underline{F}(X))$, respectively, are derived
equivalent by \cite[Theorem 1.1]{Pan2014b}.
\end{Remark}

\begin{proof} Let $U=A\oplus X$ and $V=B\oplus\underline{F}(X)$. And let $\cpx{P}_U$ and $\cpx{P_V}$ be projective resolutions of $U$ and $V$, respectively.
Then we have $\RHom_A(U,U)=\cpx{\Hom_A}(\cpx{P}_U,\cpx{P}_U)$ and $\RHom_A(V,V)=\cpx{\Hom_A}(\cpx{P}_V,\cpx{P}_V).$
By Remark \ref{5.5}, the dg endomorphism algebras $\cpx{\Hom_A}(\cpx{P_U},\cpx{P_U})$
and $\cpx{\Hom_A}(\cpx{P_V},\cpx{P_V})$ are unique under derived equivalences.
Let $\Lambda=\cpx{\Hom_A}(\cpx{P_U},\cpx{P_U})$ and $\Gamma=\cpx{\Hom_A}(\cpx{P_V},\cpx{P_V})$. It suffices to show that dg algebras $\Lambda$ and $\Gamma$ are derived equivalent.

Let $\cpx{\bar{T}}_U=F(U)$ and $\cpx{P}_{\cpx{\bar{T}}_U}$ be a projective resolution of $\cpx{\bar{T}}_U.$ Set
$\Lambda'=\cpx{\Hom_B}(\cpx{P}_{\cpx{\bar{T}}_U},\cpx{P}_{\cpx{\bar{T}}_U})$.
We claim that there is a derived equivalence between $\Lambda'$ and $\Gamma$.
In the following, we are going to prove the claim.

It follows from Lemma \ref{lemma-modCompForm} that
$\cpx{\bar{T}}_U$ is isomorphic, in $\D{\modcat{B}}$, to a complex of the form
$$0\lra \bar{T}_U^0\lra \bar{T}_U^1\lra\cdots\lra\bar{T}_U^n\lra 0$$
with $\bar{T}_U^i$ projective for $ 1\leq i\leq n$, $\bar{T}_U^0=\bar{T}^0\oplus \underline{F}(X)$ and
$\bar{T}^0$ is the first term of $\cpx{\bar{T}}$, where $\cpx{\bar{T}}$ is a tilting complex for $B$ by Lemma \ref{lemma-tiltCompForm}, we see
that $\add\cpx{\bar{T}}$ generates $\Kb{\add_{B}B}$ as
triangulated category. All the terms of $\cpx{\bar{T}}$ are in
$\add_{B}B$. From the distinguished triangle
$$\bar{T}^{+}\ra\cpx{\bar{T}}_{U}\ra\bar{T}^{0}_{U}\ra
\bar{T}_{U}^{+}[1],$$ it follows that $\bar{T}^{0}_{U}$ is in the
triangulated subcategory generated by
$\add(\cpx{\bar{T}}_{U})$. Therefore,
$\add\cpx{\bar{T}}_U$ generates $\Kb{\add_{B}V}$ as a triangulated
category. Consequently, $\thick_{\K{B}}(\cpx{\bar{T}}_U)=\thick_{\K{B}}(V).$
The definition of $\cpx{P}_{\cpx{\bar{T}}_U}$ implies $\thick_{\D{B}}(\cpx{P}_{\cpx{\bar{T}}_U})=\thick_{\D{B}}(\cpx{\bar{T}}_U).$
Similarly, $\thick_{\D{B}}(\cpx{P}_V)=\thick_{\D{B}}(V).$
The canonical functor $\K{B}\lraf{q} \D{B}$ induces an equivalence between $\thick_{\K{B}}(\cpx{P}_{\cpx{\bar{T}}_U})$ and $\thick_{\D{B}}(\cpx{P}_{\cpx{\bar{T}}_U})$ by \cite[Remark 1.7]{Hoshino2003a}.
Since $V\in\thick_{\K{B}}(V)=\thick_{\K{B}}(\cpx{\bar{T}}_U),$
Lemma \ref{thick} shows $V\in \thick_{\D{B}}(\cpx{\bar{T}}_U)$. Thus we get $\thick_{\D{B}}(V)\subseteq\thick_{\D{B}}(\cpx{\bar{T}}_U)=\thick_{\D{B}}(\cpx{P}_{\cpx{\bar{T}}_U}).$
Consequently, we obtain $\thick_{\D{B}}(\cpx{P}_V)=\thick_{\D{B}}(V)\subseteq\thick_{\D{B}}(\cpx{P}_{\cpx{\bar{T}}_U}).$
It follows from \cite[Remark 1.7]{Hoshino2003a} that $\thick_{\K{B}}(\cpx{P}_V)\subseteq \thick_{\K{B}}(\cpx{P}_{\cpx{\bar{T}}_U}).$
Similarly, $\thick_{\K{B}}(\cpx{P}_{\cpx{\bar{T}}_U})\subseteq \thick_{\K{B}}(\cpx{P}_V).$
Then there is a dg $\Gamma$-$\Lambda'$-bimodule $\cpx{\Hom_B}(\cpx{P}_{\cpx{\bar{T}}_U},\cpx{P_V})$ which induces the derived equivalence between $\Lambda'$ and $\Gamma $ by Lemma \ref{3.6}. It completes the proof of the claim.

Finally, it suffices to show that $\Lambda$ and $\Lambda'$ are derived equivalent.
Since $F$ is a standard derived equivalence, we write $F\simeq
\cpx{Y}\otimes_A^L-$, where $\cpx{Y}$ is a bounded above complex of
projective $B$-$A$-bimodules \cite{Rickard1991}. Then
$$\RHom_B(\cpx{\bar{T}}_U,\cpx{\bar{T}}_U)\simeq
\cpx{\Hom}_B(\cpx{Y}\otimes_A^L\cpx{P}_U,\cpx{Y}\otimes_A^L\cpx{P}_U)
=\cpx{\Hom}_B(\cpx{Y}\otimes_A\cpx{P}_U,\cpx{Y}\otimes_A\cpx{P}_U)\simeq\displaystyle\oplus_{n\in\mathbb{Z}}\sqcap_{i\in\mathbb{Z}}\Hom_B(M^{p},M^{p+n}).$$
We set $\cpx{M}=\cpx{Y}\otimes_A^L\cpx{P}_U$ and then
$M^p=\displaystyle\oplus_{i\in\mathbb{Z}} Y^{p-i}\otimes_A P_U^i$.
Thus, we have
$$\Hom_B(M^{p},M^{p+n})=\Hom_B(\displaystyle\oplus_{i\in\mathbb{Z}}
Y^{p-i}\otimes_A P_U^i,\displaystyle\oplus_{i\in\mathbb{Z}}
Y^{p-i}\otimes_A P_U^{n+i}).$$
Now, we can define a dg algebra
homomorphism $\cpx{\Hom_A}(\cpx{P}_U,\cpx{P}_U)\lra\cpx{\Hom_B}(\cpx{M},\cpx{M})$ defined by $(f^i)\ra (g^p),$
where $g^p=(Y^{p-i}\otimes
f^i)$. Indeed, we the above dg algebra homomorphism is a
chain map between complexes. Then we get the following commutative
diagram
$$\xymatrix{
\sqcap_{i\in\mathbb{Z}}\Hom_B(P_U^i,P_U^{i+n})\ar[d] \ar[r]^{d}
                &\sqcap_{i\in\mathbb{Z}}\Hom_B(P_U^i,P_U^{i+n+1}) \ar[d]  \\
\sqcap_{i\in\mathbb{Z}}\Hom_B(M^{p},M^{p+n+1}) \ar[r]_{d}
                & \sqcap_{i\in\mathbb{Z}}\Hom_B(M^{p},M^{p+n+1})  ,}
$$
chasing by the following diagram
$$\xymatrix{
(f^i) \ar[d]\ar[rr]^{d}
                &&(df^i-(-1)^nf^{i+1}d) \ar[d]  \\
 ((1\otimes f^i)) \ar[rr]^{d}
                &&(d(1\otimes f^i)-(-1)^n(1\otimes f^{i+1})d)  .}
$$
That is,
$df^i-(-1)^nf^{i+1}d\shortmid\lra((1\otimes d)-(-1)^n(1\otimes f^{i+1})(1\otimes
d),(d\otimes1)-(-1)^n(1\otimes f^{i})(d\otimes1)).$
Therefore, the dg algebras $\cpx{\Hom_A}(\cpx{P}_U,\cpx{P}_U)$ and
$\cpx{\Hom_B}(\cpx{M},\cpx{M})$ are quasi-isomorphic. It follows that
$\Lambda$ and $\Lambda'$ are
quasi-isomorphic. Then the dg algebras $\Lambda$ and
$\Lambda'$ are derived equivalent by Corollary \ref{quasi-iso}.
By the above argument, we conclude that the dg algebras $\Lambda$
and $\Gamma$ are derived equivalent.
\end{proof}

%\begin{Remark} Let $A$ and $B$ be finite dimensional
%$k$-algebras. Let $\mathscr{C}_{dg}(A)$ denote the dg category of
%complexes of $A$-modules. The dg category $\mathscr{C}_{dg}(A)$ has
%as objects all complexes of and its morphism are defined by
%$$\Hom_{\mathscr{C}_{dg}(A)}(\cpx{X},\cpx{Y})=\cpx{\Hom_A}(\cpx{X},\cpx{Y}),$$
%where $\cpx{\Hom_A}(\cpx{X},\cpx{Y})$ is a graded $k$-module with the component
%$$\cpx{\Hom_A}(\cpx{X},\cpx{Y})^n=\displaystyle\sqcap_{i\in\mathbb{Z}}\Hom_B(X^i,X^{i+n})$$
%and the differential defined by
%$$d(f)=(d\circ f+(-1)^{n+1}f\circ d), (f^i)\in\cpx{\Hom_A}(\cpx{X},\cpx{Y})^n.$$
%Indeed, if $\cpx{M}$ is a complex of $B$-$A$-bimodules, then
%$\cpx{M}\otimes^\bullet-$ is a dg functor between $\mathscr{C}_{dg}(B)$ and $\mathscr{C}_{dg}(A)$.
%\end{Remark}

\begin{exm}Let $A=k[x]/(x^n)$. Then $A$ is a representation-finite self-injective algebra.
Denote the indecomposable $A$-module by
$$X_r:=k[x]/(x^r)$$
 for $r=1,2,\cdots,n$.
Theorem \ref{5.2} thus shows that two dg algebras $\RHom_A(A\oplus X_r, A\oplus X_r)$ and $\RHom_A(A\oplus X_{n-r},A\oplus X_{n-r})$ are derived equivalent with $\underline{F}(X_r)=\Omega(X_r)=X_{n-r}$.
 Take the minimal projective $$\cpx{P}_{X_r}=:\cdots\lra A\lraf{f_r}A \lraf{f_{n-r}}A\lraf{f_r} A\lraf{f_{n-r}}A\lra 0$$
and
$$\cpx{P}_{X_{n-r}}=:\cdots\lra A\lraf{f_{n-r}}A \lraf{f_r}A\lraf{f_{n-r}} A\lraf{f_r}A\lra 0$$
of $X_r$ and $X_{n-r}$, respectively, where $f_r:X_n\lra X_n$  by $1+(x^n)\mapsto x^{n-r}+(x^n)$
and $f_{n-r}:X_n\lra X_n$ defined by $1+(x^n)\mapsto x^r+(x^n)$
see \cite{Zheng2015} for more details. Then
two dg algebras $\cpx{\Hom}_A(A\oplus \cpx{P}_{X_r},A\oplus \cpx{P}_{X_r})$ and $\cpx{\Hom}_A(A\oplus \cpx{P}_{X_{n-r}}, A\oplus \cpx{P}_{X_{n-r}})$ are derived equivalent.

It is known that the $\mathbb{N}$-Auslander-Yoneda algebras $\E^{\mathbb{N}}(A\oplus X_r)$ and $\E^{\mathbb{N}}(A\oplus X_{n-r})$ are derived equivalent,
and they are described in terms of quivers with relations \cite{Zheng2015}. Then we get the derived equivalence between dg algebras $\cpx{\Hom}_A(A\oplus \cpx{P}_{X_r},A\oplus \cpx{P}_{X_r})$ and $\cpx{\Hom}_A(A\oplus \cpx{P}_{X_{n-r}}, A\oplus \cpx{P}_{X_{n-r}})$,
and their cohomology rings $\E^{\mathbb{N}}(A\oplus X_r)$ and $\E^{\mathbb{N}}(A\oplus X_{n-r})$ are also derived equivalent.
\end{exm}

\bigskip
\noindent{\bf Acknowledgements.} S. Y. Pan is supported by
National Natural Science Foundation
of China (No. $11201022$), the Fundamental Research Funds for the Central Universities (2015JBM101) and the $111$ Project of China (No. B16002). The authors thank Prof. Wei Hu for his useful suggestions and discussions.

%\bibliographystyle{aomalpha}
%\bibliography{refData1}

\bigskip
Shengyong Pan

\medskip
Department of Mathematics, Beijing Jiaotong
University, 100044 Beijing, China

\smallskip
Beijing Center for Mathematics and Information Interdisciplinary Sciences, 100048 Beijing, China

\medskip
{\tt Email: shypan@bjtu.edu.cn}

\bigskip
Zhen Peng

\medskip
School of Mathematics and Statistics,
Anyang Normal University, 455000 Anyang, China

\medskip
{\tt Email: pzhen2002@163.com}

\bigskip
Jie Zhang

\medskip
School of Mathematics and Statistics, Beijing Institute of Technology, 100081 Beijing, China

\medskip
{\tt Email: jiezhang@bit.edu.cn }

\end{document}